\documentclass[a4paper,notitlepage,twoside,leqno,12pt]{amsart}
\usepackage[all]{xy}
\usepackage{framed}
\usepackage{CJK}
\usepackage{mathrsfs}
\usepackage{bbm,pifont,latexsym}
\usepackage{anysize}  
\usepackage{enumerate}
\usepackage{fancyhdr} 
\usepackage{dcolumn,indentfirst,color}
\usepackage[pdfpagemode=UseNone,pdfstartview=FitH]{hyperref}
\usepackage{amsmath,amssymb,amscd,amsthm,amsfonts,mathrsfs}
\usepackage{color,graphicx,xcolor,graphics}
\usepackage{titlesec} 
\usepackage{titletoc} 

\marginsize{3cm}{3cm}{3cm}{3cm}
\titlecontents{section}[20pt]{\addvspace{2pt}\filright}
{\contentspush{\thecontentslabel. }}
{}{\titlerule*[8pt]{}\contentspage}
\newtheorem{thm}{Theorem}[section]
\newtheorem{lem}{Lemma}[section]
\newtheorem{cor}{Corollary}[section]

\newtheorem*{thmA}{McMullen's Theorem}

\newcommand{\omC}{\widehat{\mathbb{C}}}
\newcommand{\ra}{\rightarrow}
\newcommand{\pa}{\partial}
\newcommand{\sm}{\setminus}
\newcommand{\wh}{\widehat}

\newcommand{\mc}{\mathcal}

\makeatletter\@addtoreset{equation}{section}\makeatother 

\titleformat{\section}{\large}{\textbf{\thesection.}}{1em}{\textbf}
\titleformat{\subsection}{\normalsize}{\textbf{\thesubsection.}}{1em}{\textbf}
\titleformat{\subsubsection}{\normalsize}{\thesubsubsection.}{1em}{\textbf}

\begin{document}
\title{Wandering Julia components \\ of cubic rational maps}
\author{Guizhen Cui}
\address{Guizhen Cui, School of Mathematical Sciences, Shenzhen University, Shenzhen, 518052, P. R. China, \& Academy of Mathematics and Systems Science Chinese Academy of Sciences, Beijing, 100190, P. R. China}
\email{gzcui@szu.edu.cn}
\author{Wenjuan Peng}
\address{Wenjuan Peng, Academy of Mathematics and Systems Science Chinese Academy of Sciences, Beijing, 100190, P. R. China}
\email{wenjpeng@amss.ac.cn}
\author{Luxian Yang}
\address{Luxian Yang, School of Mathematical Sciences, Shenzhen University, Shenzhen, 518052, P. R. China}
\email{lxyang@szu.edu.cn}
	
\thanks{The first author is supported by National Key R\&D Program of China no.2021YFA1003203, and the NSFC Grants No. 12131016 and No. 12071303. The second author is supported by the NSFC Grants No. 12122117 and No. 12288201. The third author is supported by Basic and Applied Basic Research Foundation of Guangdong Province Grants No.2023A1515010058. }
\subjclass[2000]{}

\begin{abstract}
We prove that every wandering Julia component of cubic rational maps eventually has at most two complementary components.
\end{abstract}
\maketitle

\section{Introduction}
Let $f$ be a rational map of the Riemann sphere $\omC$ with degree $\deg f\geq2$. The Fatou set $F_f$ of $f$ is defined as the point $z\in\omC$ whose iteration sequence $\{f^n\}_{n =1}^{\infty}$ is a normal family in a neighborhood of $z$. The Julia set $J_f$ of $f$ is the complement of the Fatou set. Refer to \cite{milnor2006dynamics,beardon1991iteration} for their basic properties.

A connected component of $F_f$ is called a Fatou domain. A Fatou domain is eventually periodic if its forward orbit has finitely many components. Otherwise it is wandering. The well-known Sullivan Non-wandering Theorem says that every Fatou domain of a rational map is eventually periodic \cite{sullivan1985quasiconformal}. Moreover, the connectivity of a periodic Fatou domain equals $1,2$ or $\infty$. By using quasiconformal surgeries, Shishikura proved that the number of periodic cycles of Fatou domains is at most $2\deg f-2$ \cite{shishikura1987quasiconformal}.

Similarly, a connected component of $J_f$ is eventually periodic if its forward orbit has finitely many components, and wandering otherwise. There are countably many eventually periodic Julia components and uncountably many wandering Julia components (See \cite{milnor2006dynamics, mcmullen1988automorphisms, pilgrim2000rational}). For each non-trivial periodic Julia component $K$ of $f$ with period $p\ge 1$, McMullen proved that there exists a rational map $g$ of $\deg g\ge 2$ such that$(f^p,K)$ is quasiconformally conjugate to $(g,J_g)$ \cite{mcmullen1988automorphisms}. So the study of periodic Julia components is in some sense converted to the study of connected Julia sets.

For wandering Julia components, some results have been obtained for special classes of rational maps. Branner and Hubbard proved that wandering Julia components of cubic polynomials are points \cite{branner1992iteration}. This result is extended to polynomials of any degree by Yin and Qiu, Kozlovski and van Strien independently \cite{qiu2009proof, kozlovski2009local}. For rational maps, McMullen constructed examples of hyperbolic rational maps with wandering Jordan curves as Julia components \cite{mcmullen1988automorphisms}. In fact, Pilgrim and Tan proved that for geometrically finite rational maps, every wandering Julia components are either points or Jordan curves \cite{pilgrim2000rational}.

For general rational maps, the first attempt to study the topology of wandering Julia components was done by Cui, Peng and Tan \cite{cui2011topology}. They proved that if a rational map has at most one full wandering Julia component containing critical values, then every wandering Julia component eventually has at most two complementary components. Here a compact set in $\omC$ is {\bf full} if its complement is connected. In this paper, we study the wandering Julia components of cubic rational maps. Our main result is

\begin{thm}\label{thm:main}
Let $f:\omC\ra\omC$ be a cubic rational map with disconnected Julia set. Suppose that $K$ is a wandering Julia component. Then for sufficiently large integer $n\ge 0$, $\omC\sm f^n(K)$ has at most two components.
\end{thm}

It is known that wandering Julia components of quadratic rational maps must be points \cite{milnor2000rational}, and cubic rational maps may have wandering Jordan curves as Julia components \cite{godillon2015family}.

Our proof is based on the argument in \cite{cui2011topology}, where they studied the puzzle pieces formed by a wandering Julia component whose iterated images always have at least three complementary components. By developing this argument, we prove that if each full wandering Julia component contains at most one critical value, the conclusion in Theorem \ref{thm:main} is true. On the other hand, we prove that a rational map with disconnected Julia set has at least two critical values in the union of the Fatou set together with the periodic Julia components. By enumerating the number of critical values, we complete the proof of Theorem \ref{thm:main}.

\section{Configurations of wandering Julia components}
We will prove the next result in this section.

\begin{thm}\label{thm:2}
Let $f$ be a rational map with disconnected Julia set. Suppose that each full wandering Julia component of $f$ contains at most one critical value. Then for any wandering Julia component $K$ of $f$, $\omC\sm f^n(K)$ has at most two components as $n\ge 0$ is large enough.
\end{thm}

Here we mention that if a Julia component $K$ is full, then $f^n(K)$ is full for all $n\ge 1$. If $\omC\sm f(K)$ has $k\ge 2$ components, then $\omC\sm K$ has at least $k\ge 2$ components. It is known that there are examples of rational maps with wandering Julia components $K$ which have more than two complementary components, but $\omC\sm f^n(K)$ has at most two components as $n\ge 0$ is large enough \cite{cui2011topology}.

\vskip 0.24cm
The proof of Theorem \ref{thm:2} is based on the argument in \cite{cui2011topology} which is stated in the following. Let $f$ be a rational map with disconnected Julia set. Since $f$ always has fixed points, we may assume that $f(\infty)=\infty$. Let $V_f$ be the set of critical values of $f$ and $X=V_f\cup\{\infty\}$. Denote by
\begin{itemize}
\item $\mc{C}$: the collection of wandering Julia components $K$ of $f$ such that $\omC\sm f^n(K)$ has at least three components for all $n\ge 0$.
\item $\mc{S}\subset\mc{C}$: $K\in\mc{S}$ if either $K$ intersects $X$ or $\omC\sm K$ has at least three components intersects $X$.
\item $\mc{D}\subset\mc{C}$: $K\in\mc{D}$ if $f^n(K)\in\mc{C}\sm\mc{S}$ for all $n\ge 0$ and $\omC\sm K$ has exactly one component intersecting $X$.
\item $\mc{A}\subset\mc{C}$: $K\in\mc{A}$ if $f^n(K)\in\mc{C}\sm\mc{S}$ for all $n\ge 0$ and $\omC\sm K$ has exactly two components intersecting $X$.
\end{itemize}

It is clear that $\#\mc{S}\le 2\# X-2$. Thus for any $K\in\mc{C}$, $f^n(K)\in\mc{C}\sm\mc{S}$ as $n\ge 0$ is large enough. Hence $\#(\mc{D}\cup\mc{A})=\infty$ if $\mc{C}\neq\emptyset$. By definition, $f(\mc{D}\cup\mc{A})\subset\mc{D}\cup\mc{A}$.

\vskip 0.24cm
For any $K\in\mc{S}$, denoted by $\wh{K}$ the union of $K$ together with all bounded components of $\omC\sm K$. Then $\wh{K}$ is a bounded full continuum. If $K\in\mc{D}$, then $\wh{K}$ is disjoint from $V_f$.

\begin{lem}\label{lem:A}
For any $K\in\mc{C}$, $\{f^n(K)\}_{n\ge 0}$ visits $\mc{A}$ infinitely many times.
\end{lem}

\begin{proof}
Assume by contradiction that $\{f^n(K)\}_{n\ge 0}$ visits $\mc{A}$ only finitely many times. Then there is an integer $n_0\ge 0$ such that $f^n(K)\in\mc{D}$ for all $n\ge n_0$. Denote $K_n=f^n(K)$. Then $f:\,\wh{K}_n\to\wh{K}_{n+1}$ is a homeomorphism for $n\geq n_0$ since $\wh{K}_{n+1}$ is a full continuum disjoint from $V_f$. We claim that for any $m,n\geq n_0$ with $m\neq n$, $\wh{K}_n\cap\wh{K}_m=\emptyset$. Otherwise, one is contained in another since $K_n\cap K_m=\emptyset$. We may assume that $\wh{K}_n\subset\wh{K}_m$. Then there is a component $\Omega$ of $\omC\sm K_m$ such that $K_n\subset\Omega$. Hence there is an interger $k\ge 1$ such that $J_f\subset f^{k}(\Omega)$ (refer to \cite[Corollary 14.2]{milnor2006dynamics}). Since for any $n\ge 1$, $f:\,\wh{K}_n\to\wh{K}_{n+1}$ is a homeomorphism, we know that $f^{k}(\Omega)$ is a bounded component of $\omC\sm K_{m+k}$. This contradicts the fact that $K_{m+k}$ is a Julia component.

By the claim, $\{f^n\}$ is a normal family on each bounded component of $\omC\sm K_{n_0}$, which is a wandering Fatou domain of $f$. This contradicts the Sullivan Non-wandering Theorem. Now the lemma is proved.
\end{proof}

By Lemma \ref{lem:A}, $\#\mc{A}=\infty$ if $\mc{C}\neq\emptyset$. Define an equivalence relation on $\mc{A}$ by $K\sim K'$ if $\wh{K}\cap X=\wh{K'}\cap X$; and an order relation by $K\prec K'$ if $K\sim K'$ and $\wh{K'}\subset\wh{K}$. If $K\sim K'$ but $K\neq K'$, then either $\wh{K}\subset\wh{K'}$ or $\wh{K'}\subset\wh{K}$. Thus either $K'\prec K$ or $K\prec K'$. 

Define the first return map $F$ on $\mc{A}$ by $F(K)=f^{n(K)}(K)$, where $n(K)\ge 1$ is the unique integer such that $f^{n(K)}(K)\in\mc{A}$ but $f^n(K)\in\mc{D}$ for $0<n<n(K)$. Then $f^{n(K)-1}$ is a homeomorphism on $\wh{K}$. By definition, $n(K')\le n(K)$ if $K'\prec K$.

For any $K_0\in\mc{A}$, since $\#X<\infty$, according to the equivalence relation, the orbit $\{F^n(K_0)\}_{n\ge 0}$ is decomposed into finitely many equivalence classes. Denote by $\mc{A}_i(K_0)$ ($1\le i\le q$) the equivalence classes which $\{F^n(K_0)\}_{n\ge 0}$ visits infinitely many times. Then there is an integer $n_0\geq 0$ such that for all $n\geq n_0$, $F^n(K_0)\in\bigcup_{1\le i\le q}\mc{A}_i(K_0).$ To simplify the notations, we may assume that $n_0=0$. Denote by $V(K)$ the unique bounded component of $\omC\sm K$ intersecting $V_f$ for $K\in\mc{A}$. Set
$$
E_i(K_0)=\bigcap_{K\in\mc{A}_i(K_0)}V(K)\,\text{ and }\,N_i(K_0)=\sup_{K\in\mc{A}_i(K_0)}n(K).
$$

\begin{lem}\label{lem:full}
If $N_i(K_0)=\infty$, then $E_i(K_0)$ is either a full wandering Julia component or a single point.
\end{lem}

\begin{proof}
Since $N_i(K_0)=\infty$, for any $K\in\mc{A}_i(K_0)$, there exists $K'\in\mc{A}_i(K_0)$ such that $n(K')>n(K)$. Thus $K\prec K'$ and hence $V(K')$ is compactly contained in $V(K)$. Therefore $E_i(K_0)$ is a full continuum containing critical values of $f$.

For any $K\in\mc{A}_i(K_0)$, since $f^{n(K)-1}$ is injective on $V(K)$ and $N_i(K_0)=\infty$, we known that $f^n$ is injective on $E_i(K_0)$ for all $n\ge 0$. We claim that $E_i(K_0)\subset J_f$. Otherwise $E_i(K_0)$ contains a Fatou domain and hence $f^k(E_i(K_0))$ contains a periodic Fatou domain $D$ of period $p\ge 1$ for some integer $k\ge 0$. If $D$ is not a rotation domain, then $\deg f^p|_D>1$. This contradicts the fact that $f^n$ is injective on $E_i(K_0)$ for all $n\ge 0$. If $D$ is a rotation domain, then $\partial D$ is contained in a non-trivial periodic Julia component $J$ of $f$ with $f^p(J)=J$. In particular, $J\subset E_i(K_0)$ and $\deg f^p|_J>1$ (refer to McMullen's Theorem in \S3). This is also a contradiction. Now the claim is proved.

By the claim $E_i(K_0)$ is contained in a Julia component $E_i'$. On the other hand, as a Julia component, $E_i'\subset V(K)$ for all $K\in\mc{A}_i(K_0)$. Thus $E_i(K_0)=E_i'$ is a full Julia component. If $E_i(K_0)$ is eventually periodic, then it must be a single point since $f^n$ is injective on $E_i(K_0)$ for all $n\ge 0$.
\end{proof}

For any $K\in\mc{A}$, denote $K'=f^{n(K)-1}(K)$. Then $f(K')=F(K)$. Since $\wh{F(K)}\sm V(F(K))$ is disjoint from $V_f$, the component of $f^{-1}(\wh{F(K)}\sm V(F(K)))$ containing $K'$ has exactly two complementary components, and one is unbounded since $f(\infty)=\infty$. Let $W$ be the other one. Then $W$ is a bounded component of $\omC\sm K'$ and $f:\, \wh{K'}\sm W\to\wh{F(K)}\sm V(F(K))$ is proper. Since $f^{n(K)-1}: \wh{K}\to\wh{K'}$ is a homeomorphism, there is a unique bounded component of $\omC\sm K$, denoted by $U(K)$, such that $f^{n(K)-1}: \wh{K}\sm U(K)\to\wh{K'}\sm W$ is a homeomorphism. Thus $f^{n(K)}:\wh{K}\sm U(K)\to\wh{F(K)}\sm V(F(K))$ is proper.


\begin{lem}\label{lem:UnV}
For any $K\in\mc{A}$, there exists an integer $k\ge 1$ such that $U(F^k(K))\neq V(F^k(K))$.
\end{lem}

\begin{proof}
Assume by contradiction that for any $k\ge 1$, $U(F^k(K))=V(F^k(K))$. For any bounded component $\Omega$ of $\omC\sm K$ with $\Omega\neq U(K)$, $f^n(\Omega)$ is a bounded component of $\omC\sm f^n(K)$ for all $n\ge 1$. Thus $\Omega$ is a wandering Fatou domain. This is a contradiction.
\end{proof}

Denote
$$
\mc{B}=\{K\in\mc{A}:U(K)\neq V(K)\}.
$$
Then for any $K\in\mc{A}$, $\{F^n(K)\}_{n\ge 0}$ visits $\mc{B}$ infinitely many times by Lemma \ref{lem:UnV}. Thus $\#\mc{B}=\infty$ if $\mc{C}\neq\emptyset$. Define the first return map $G$ on $\mc{B}$ by $G(K)=f^{m(K)}(K)$, where $m(K)\geq 1$ is the unique integer such that $f^{m(K)}(K)\in\mc{B}$ but $f^n(K)\notin\mc{B}$ for $0<n<m(K)$. Then
$$
f^{m(K)}:\wh{K}\sm U(K)\to\wh{G(K)}\sm V(G(K))
$$
is proper. Note that $G=F^k$ on $K$ for some integer $k\ge 1$.

\begin{lem}\label{lem:order}
For any $K\in\mc{B}$,
\begin{itemize}
\item[(1)] if $K'\in\mc{B}$ and $K\prec K'$, then $m(K)<n(K')\le m(K')$;
\item[(2)] $\{K'\in\mc{B}:K'\prec K\}$ is a finite collection;
\item[(3)] if $G(K)\sim K$, then $G(K)\prec K$.
\end{itemize}
\end{lem}

\begin{proof}
(1) The fact $n(K')\le m(K')$ follows directly from the definitions. Assume that $G=F^k$ on $K$ with $k\ge 1$. Then $U(K)\neq V(K)$ and $U(F^i(K))=V(F^i(K))$ for $1\le i<k$. Denote $m_0=n(K)$ and $m_i=n(K)+\cdots+n(F^i(K))$ for $1\le i<k$. Then $m(K)=m_{k-1}$. Since $V(K)\neq U(K)$, $f^{n(K)}(V(K))$ is a bounded component of $\omC\sm F(K)$ disjoint from $V(F(K))$. If $k=1$, then $m(K)=n(K)$. So $f^{m(K)}(V(K))$ is a bounded component of $\omC\sm F(K)$ disjoint from $V(F(K))$. If $k>1$, then $f^{n(K)}(V(K))\neq U(F(K))$. Therefore $f^{m_1}(V(K))$ is a bounded component of $\omC\sm F^2(K)$ disjoint from $V(F^2(K))$. Inductively, $f^{m(K)}(V(K))=f^{m_{k-1}}(V(K))$ is a bounded component of $\omC\sm F^{k}(K)$ disjoint from $V(F^{k}(K))$ (see Figure \ref{fig:G3}). In summary, $f^{j}(V(K))$ is a bounded component of $\omC\sm f^j(K)$ disjoint from $V_f$ for $1\le j\le m(K)$.

Because $K'\subset V(K)$, $f^j(K')\subset f^j(V(K))$ for $1\le j\le m(K)$. Thus $f^j(K')\in\mc{D}$ for $1\le j\le m(K)$ and hence $m(K)<n(K')$.

\begin{figure}[htbp]
\centering
\includegraphics[width=14.5cm]{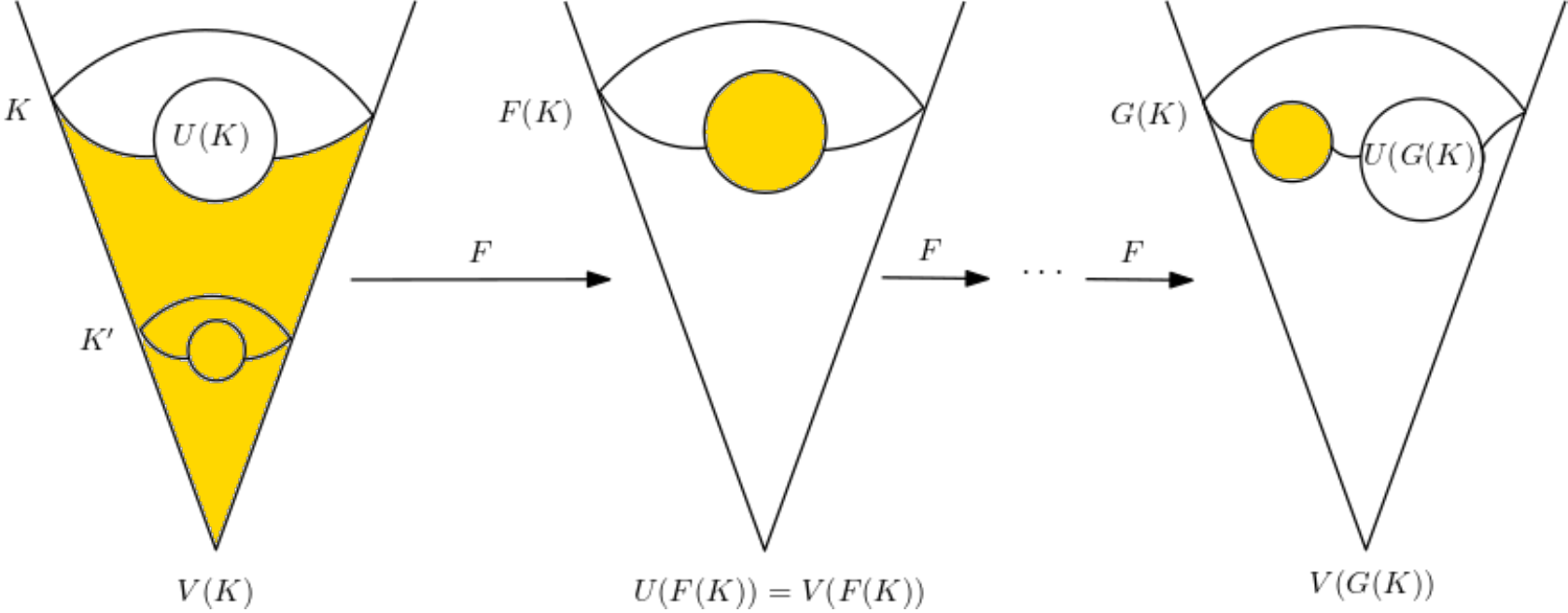}
\caption{The iterated images of $V(K)$ under $f$. }
\label{fig:G3}
\end{figure}

(2) This a direct consequence of (1) since $m(K')<\infty$.

(3) Assume by contradiction that $K\prec G(K)$. Then $V(K)\supset\wh{G(K)}$. From (1) we have $m(K)<m(G(K))$. Thus $f^{m(K)}(V(K))$ is a bounded component of $\omC\sm G(K)$ and hence $f^{m(K)}(V(K))\subset V(K)$ (see Figure \ref{fig:G1}). Therefore the sequence $\{f^{k\cdot m(K)}\}_{k\ge 1}$ is bounded on $V(K)$ and hence forms a normal family. So $V(K)\subset F_f$. This contradicts $G(K)\subset V(K)$.
\end{proof}

\begin{figure}[htbp]
\centering
\includegraphics[width=5cm]{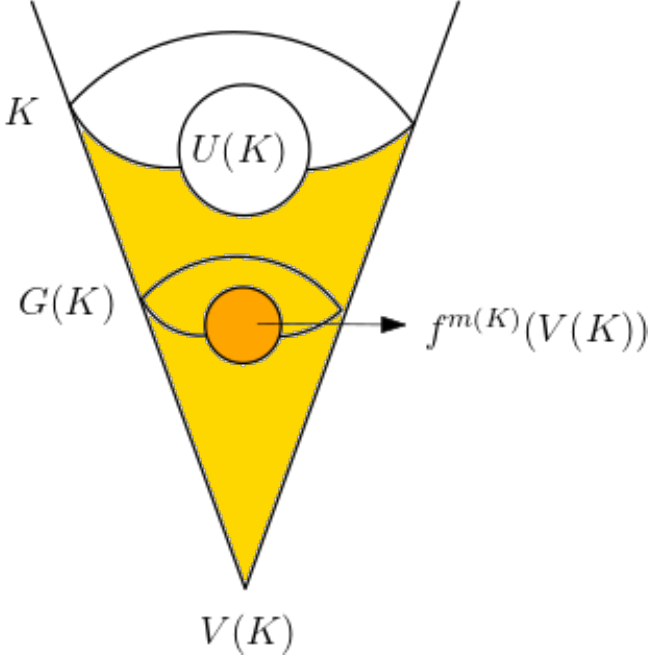}
\caption{$K\prec G(K)$.}
\label{fig:G1}
\end{figure}

\begin{lem}\label{lem:full1}
Assume $\mc{A}\neq\emptyset$. Then for any $K_0\in\mc{A}$, $\{F^n(K_0)\}_{n\ge 0}$ contains at least two equivalence classes $\mc{A}_i(K_0)$ such that $\#(\mc{A}_i(K_0)\cap\mc{B})=\infty$. Moreover,
$$
E_i(K_0)=\bigcap_{K\in\mc{A}_i(K_0)}V(K)
$$
is a full wandering Julia component if $\#(\mc{A}_i(K_0)\cap\mc{B})=\infty$.
\end{lem}

\begin{proof}
Assume by contradiction that there is only one equivalence class $\mc{A}_i(K_0)$ such that $\#(\mc{A}_i(K_0)\cap\mc{B})=\infty$. Then there is an element $K\in\mc{A}_i(K_0)\cap\mc{B}$ such that $G^n(K)\in\mc{A}_i(K_0)\cap\mc{B}$ for all $n\ge 0$. By (3) of Lemma \ref{lem:order}, $G^{n+1}(K)\prec G^{n}(K)$ for all $n\ge 0$. Thus $G^n(K)\prec K$ for all $n\ge 1$. This contradicts (2) of Lemma \ref{lem:order}.

Suppose that $\#(\mc{A}_i(K_0)\cap\mc{B})=\infty$. Then there is an infinite increasing sequence $\{k_j\}_{j\ge 1}$ such that $f^{k_j}(K_0)\in\mc{A}_i(K_0)\cap\mc{B}$. By (2) of Lemma \ref{lem:order}, by passing to a subsequence, we may assume that $f^{k_j}(K_0)\prec f^{k_{j+1}}(K_0)\text{ for $j\ge 1$}$. Then
$$
E_i(K_0)=\bigcap_{j\ge 1}V(f^{k_j}(K_0)).
$$
Set $m_j=m(f^{k_j}(K_0))$. By (1) of Lemma \ref{lem:order}, $\{m_j\}$ is strictly increasing and
$$
N_i(K_0)=\sup_{K\in\mc{A}_i(K_0)}n(K)=\infty.
$$
Thus $E_i(K_0)$ is either a full wandering Julia component or a single point by Lemma \ref{lem:full}.

Note that $f^{m_j}(f^{k_j}(K_0))=f^{m_j+k_j}(K_0)$ are pairwise disjoint since $\{m_j+k_j\}$ is strictly increasing. Therefore $f^{m_j}(V(f^{k_j}(K_0)))\neq V(f^{m_j+k_j}(K_0))$ are pairwise disjoint bounded components of $\omC\sm f^{m_j+k_j}(K_0)$. From $E_i(K_0)\subset V(f^{k_j}(K_0))$, we obtain $f^{m_j}(E_i(K_0))\subset f^{m_j}(V(f^{k_j}(K_0)))$. Thus $\{f^{m_j}(E_i(K_0))\}_{j\geq 1}$ are also pairwise disjoint. So $E_i(K_0)$ is a full wandering Julia component.
\end{proof}

The next result is a direct consequence of Lemma \ref{lem:full1}.

\begin{cor}\label{cor:onefull}
Let $f$ be a rational map with disconnected Julia set. If there is at most one full wandering Julia component containing critical values of $f$, then for any wandering Julia component $K$ of $f$, $\omC\sm f^n(K)$ has at most two components as $n\ge 0$ is large enough.
\end{cor}

\begin{lem}\label{lem:full2}
Suppose that each full wandering Julia component of $f$ contains at most one critical value. Then for any $K_0\in\mc{A}$, $E_i(K_0)$ is either a full wandering Julia component or a single point.
\end{lem}

\begin{proof}
The index set $I:=\{1,\cdots,q\}$ can be written as a disjoint union $I=I_0\cup I_1$ such that $j\in I_1$ if and only if $E_j(K_0)$ is either a full wandering Julia component or a single point. By Lemma \ref{lem:full1}, $I_1$ is non-empty. Assume by contradiction that $I_0$ is non-empty. Then $N_i(K_0)<\infty$ for $i\in I_0$ by Lemma \ref{lem:full}.

\vskip 0.24cm
\noindent{\bf Claim 1}. There exist $(i,j)\in I_0\times I_1$ and a sequence $\{K_k\}_{k\ge 1}$ in $\mc{A}_i(K_0)$ such that $F(K_k)\in\mc{A}_j(K_0)$ and
$$
\bigcap_{k\ge 1}V(F(K_k))=E_j(K_0).
$$
By Claim 1, $f^{n(K_k)}(\wh{K_k})=\wh{F(K_k)}$ since $E_j(K_0)$ contains exactly one critical value of $f$. Since $n(K_k)\le N_i(K_0)<\infty$, passing to a subsequence, we may assume that $n(K_k)=N$ is a constant for all $k\ge 1$, i.e., $f^{N}(\wh{K_k})=\wh{F(K_k)}$ for all $k\ge 1$. Thus
$$
\bigcap_{k\ge 1}f^{N}(\wh{K_k})=\bigcap_{k\ge 1}\wh{F(K_k)}=E_j(K_0).
$$
So we have $f^N(E_i(K_0))\subset E_j(K_0)$.

Let $E_i'$ be the component of $f^{-N}(E_j(K_0))$ containing $E_i(K_0)$. Then $E_i'$ is either a full wandering Julia component or a single point. Thus $E'_i\subset V(K)$ for all $K\in\mc{A}_i(K_0)$. So $E'_i\subset E_i(K_0)$ and hence $E_i(K_0)=E'_i$ is either a full wandering Julia component or a single point. This contradicts the fact $i\in I_0$.
\end{proof}

\begin{proof}[Proof of Claim 1] Assume by contradiction that the claim is not true. Then for each $j\in I_1$, there exists an $L_j\in\mc{A}_j(K_0)$ such that $F(K)\prec L_j$ for any $K\in\bigcup_{i\in I_0}\mc{A}_i(K_0)$ with $F(K)\in\mc{A}_j(K_0)$.

Define a relation in $I_1$ by $j\rhd j'$ if there is a $K\in\mc{A}_j(K_0)$ with $K\prec L_j$ and an integer $k\ge 1$ such that
\begin{itemize}
\item $F^{k}(K)\in\mc{A}_{j'}(K_0)$, $L_{j'}\prec F^{k}(K)$ or $L_{j'}=F^{k}(K)$, and
\item $F^s(K)\in\bigcup_{j\in I_1}\mc{A}_j(K_0)$ for $1\le s\le k$.
\end{itemize}
Let $m\ge 1$ be an integer such that $f^{m}=F^{k}$ on $K$. Then $f^{m}(\wh{K})=\wh{F^k(K)}$ since $\wh{F^s(K)}$ contains exactly one critical value of $f$ for $1\le s\le k$. Consequently,
$$
f^{m}(\wh{L_j})\subset f^{m}(\wh{K})=\wh{F^k(K)}\subset\wh{L_{j'}}.
$$

\vskip 0.24cm
\noindent{\bf Claim 2}.
There exists a $j_0\in I_1$ such that $j\rhd j_0$ does not hold for any $j\in I_1$.

Otherwise, there is a sequence $\{j_1,\cdots,j_{p}\}$ in $I_1$ such that
$$
j_1\rhd\cdots\rhd j_p\rhd j_{p+1}:=j_1.
$$
By definition, for $1\le t\le p$, there exist $K_t\in\mc{A}_{j_t}(K_0)$ with $K_t\prec L_{j_t}$ and integers $k_t\ge 1$ such that,
\begin{itemize}
\item $F^{k_t}(K_t)\in\mc{A}_{j_{t+1}}(K_0)$, $L_{j_{t+1}}\prec F^{k_t}(K_t)$ or $L_{j_{t+1}}=F^{k_t}(K_t)$, and
\item $f^{m_t}(\wh{L_{j_t}})\subset\wh{L_{j_{t+1}}}$, where $m_t\ge 1$ is an integer such that $f^{m_t}=F^{k_t}$ on $K$.
\end{itemize}
Thus $f^m(\wh{L_{j_1}})\subset\wh{L_{j_1}}$ for $m=m_1+\cdots+m_p$. This is a contradiction. Now Claim 2 is proved.

\vskip 0.24cm
Since $I_0$ is non-empty, there is an integer $n_0\ge 1$ such that $F^{n_0}(K_0)\in\bigcup_{i\in I_0}\mc{A}_i(K_0)$. For any $n>n_0$ with $F^n(K_0)\in\mc{A}_{j_0}(K_0)$, let $n_1\in [n_0, n)$ be the largest integer such that $F^{n_1}(K_0)\in\bigcup_{i\in I_0}\mc{A}_i(K_0)$. Then $F^{n_1+1}(K_0)\in\mc{A}_{j'}(K_0)$ for some $j'\in I_1$ and $F^{n_1+1}(K_0)\prec L_{j'}$ by the assumption. In the case $n_1+1=n$, we have $j'=j_0$ and $F^{n}(K_0)\prec L_{j_0}$. In the case $n_1+1<n$, if $L_{j_0}\prec F^n(K_0)$ or $L_{j_0}=F^n(K_0)$, then $j'\rhd j_0$. This contradicts the definition of $j_0$. Thus in both cases  $F^{n}(K_0)\prec L_{j_0}$ for all $n>n_0$ with $F^n(K_0)\in\mc{A}_{j_0}(K_0)$. So $E_{j_0}(K_0)$ is an open set. This is a contradiction. Now Claim 1 is proved.
\end{proof}

\begin{proof}[Proof of Theorem \ref{thm:2}]
Assume by contradiction that there is a wandering Julia component of $f$ such that its iterated images always have at least three complementary components. Then $\mc{C}\neq\emptyset$ and $\#\mc{A}=\infty$.

For any $K_0\in\mc{A}$, applying Lemma \ref{lem:full2}, each $E_i(K_0)$ contains exactly one critical value. Thus there is an integer $n_0\ge 1$ such that $V(F^n(K_0))$ contains exactly one critical value for $n\ge n_0$. Write $K=F^{n_0}(K_0)$ and
$$
m_k=n(K)+n(F(K))+\cdots+n(F^{k-1}(K)).
$$
Then $f^{m_{k}}:\, \wh{K}\to\wh{F^k(K)}$ is proper for all $k\ge 1$. This is a contradiction.
\end{proof}

\noindent{\bf Remark}. Comparing with the argument in \cite{cui2011topology}, besides of some little modifications, Lemmas \ref{lem:full} and \ref{lem:full2} are new developments.

\section{Critical orbits in periodic components}
In \S2, we obtain some results about the orbits of critical points in wandering Julia components. In this section, we study the critical orbits in Fatou domains and periodic Julia components.

\begin{thm}\label{thm:disconnected}
Let $f$ be a rational map with disconnected Julia set. Then either the Fatou set $F_f$ or periodic Julia components of $f$ contain at least two critical values.
\end{thm}

At first, we recall McMullen's Theorem for periodic Julia components. A Julia component $K$ is {\bf non-trivial} if it is not a single point.

\begin{thmA}
Let $f$ be a rational map with disconnected Julia set. Suppose that $K$ is a non-trivial periodic Julia component of $f$ with period $p\ge 1$. Then there exists a rational map $g$ with $\deg g\ge 2$ and a quasiconformal map $\phi:\omC\ra\omC$ such that
\begin{itemize}
\item[(1)] $\phi(K)=J_g$,
\item[(2)] $\phi\circ f^p(z)=g\circ\phi(z)$ for $z\in K$, and
\item[(3)] if $K$ intersects boundaries of Siegel discs or Herman rings, then $g$ has Siegel discs.
\end{itemize}
\end{thmA}

We mention that statement (3) is not listed in the original version of McMullen's Theorem \cite[Theorem 3.4]{mcmullen1988automorphisms}, but the proof there actually includes a proof of statement (3).

\begin{lem}\label{lem:herman}
Let $f$ be a rational map with a cycle of Herman rings. Then there are at least two distinct cycles of Julia components containing critical values of $f$.
\end{lem}

\begin{proof}
Let $H$ be a Herman ring of $f$ with period $p\ge 1$. Then there is an integer $0\le i<p$ such that $f^i(H)$ does not separate $\mc{H}\sm f^i(H)$, where $\mc{H}=\bigcup_{0\le n<p}f^n(H)$.
Assume $i=0$ for simplicity, i.e., $\mc{H}\sm H$ is contained in one component of $\omC\sm H$.

Let $B_0$ and $B_1$ be the two components of $\partial H$ such that $B_0$ does not separate $\mc{H}$. Let $K_0$ and $K_1$ be the Julia components containing $B_0$ and $B_1$, respectively. Then both of them are non-trivial periodic Julia components.

Since $K_0$ is disjoint from $\pa f^n(H)$ for all $1\le n<p$, $f^n(K_0)$ is disjoint from $\pa f^m(H)$ if $m-n\not\equiv 0\bmod p$. Otherwise $K_0=f^{p-n}(f^n(K_0))$ intersects
$f^{p-n}(\pa f^m(H))=\pa f^{p-n+m}(H)$ for some integer $m\ge 0$ with $m-n\not\equiv 0\bmod p$. This is a contradiction.

From the above discussion, we know that if $K_1$ intersects $\pa(\mc{H}\sm H)$, then $K_0$ and $K_1$ are not contained in the same cycle of Julia components.

Now we assume that $K_1$ is disjoint from $\pa(\mc{H}\sm H)$. Then $K_1$ is disjoint from $\bigcup_{1\le n<p}f^n(B_0)$. If $K_1=f^{n_0}(K_0)$ for some integer $1\le n_0<p$, then $K_1$ contains $f^{n_0}(B_0)$. This is a contradiction. Thus in both cases, $K_0$ and $K_1$ are not contained in the same cycle of Julia components.

Applying McMullen's Theorem for $K_i$ ($i=0,1$). We obtain rational maps $g_i$ with Siegel disks. Thus $J_{g_i}$ contains critical values of $g_i$ \cite[Theorem 11.17]{milnor2006dynamics}. This implies that the orbit of $K_i$ contains critical values of $f$.
\end{proof}

\begin{proof}[Proof of Theorem \ref{thm:disconnected}]
If $f$ has no Herman rings and the Fatou set $F_f$ contains only one critical value, then every Fatou domain is simply connected. Thus $J_f$ is connected. This is a contradiction.
Thus either $f$ has Herman rings or $F_f$ contains at least two critical values. In the former case, there are two distinct cycles of Julia components, and both of them contain critical values of $f$ by Lemma \ref{lem:herman}.
\end{proof}

\begin{proof}[Proof of Theorem \ref{thm:main}]
Let $f$ be a cubic rational map with disconnected Julia set. Then $\#V_f\le 4$. By Theorem \ref{thm:disconnected}, either $F_f$ or periodic Julia components containing at least two critical values of $f$. Thus there are at most two critical values of $f$ contained in wandering Julia components. In other words, either every wandering Julia component contains at most one critical value, or $f$ has at most one full wandering Julia component containing critical values. By Theorem \ref{thm:2} and Corollary \ref{cor:onefull}, for any wandering Julia component $K$ of $f$, $\omC\sm f^n(K)$ has at most two components as $n\ge 0$ is large enough.
\end{proof}

\bibliography{wandering}{}
\bibliographystyle{abbrv}
\end{document}